\documentclass[12pt,a4paper]{amsart}
\usepackage[top=35mm, bottom=30mm, left=30mm, right=30mm]{geometry}
\usepackage{mathptmx}
\usepackage{mathrsfs}
\usepackage{verbatim}
\usepackage{url}
\usepackage[all]{xy}
\usepackage{xcolor}
\usepackage[colorlinks=true,citecolor=blue]{hyperref}
\usepackage{amsmath,amsthm}
\usepackage{amssymb}
\usepackage{enumerate}
\usepackage{graphicx}

\newtheorem{thm}{Theorem}[section]

\newtheorem{lem}[thm]{Lemma}

\theoremstyle{definition}
\newtheorem{defin}[thm]{Definition}
\newtheorem{rem}[thm]{Remark}

\numberwithin{equation}{section}
\begin{document}


\baselineskip=20pt


\title{Can points of bounded orbits surround points of unbounded orbits ?}

\author[J.~Mai]{Jiehua Mai}
\address{School of Mathematics and Quantitative
Economics, Guangxi University of Finance and Economics, Nanning, Guangxi, 530003, P. R. China \&
 Institute of Mathematics, Shantou University, Shantou, Guangdong, 515063, P. R. China}
\email{jiehuamai@163.com; jhmai@stu.edu.cn}

\author[E.~Shi]{Enhui Shi*}
\thanks{*Corresponding author}
\address{School of Mathematics and Sciences, Soochow University, Suzhou, Jiangsu 215006, China}
\email{ehshi@suda.edu.cn}

\author[K.~Yan]{Kesong Yan}
\address{School of Mathematics and Statistics, Hainan Normal
University,  Haikou, Hainan, 571158, P. R. China}
\email{ksyan@mail.ustc.edu.cn}

\author[F.~Zeng]{Fanping Zeng}
\address{School of Mathematics and Quantitative
Economics, Guangxi University of Finance and Economics, Nanning, Guangxi, 530003, P. R. China}
\email{fpzeng@gxu.edu.cn}

\begin{abstract}
We show a somewhat surprising result: if $E$ is a disk in the plane $\mathbb R^2$, then there is a
homeomorphism $h:\mathbb R^2\rightarrow\mathbb R^2$ such that, for every $x\in\partial E$, the orbit
$O(x, h)$ is bounded, but for every $y\in {\rm Int}(E)$, the orbit $O(y, h)$ is doubly divergent.
To prove this, we define a class of homeomorphisms on the square $[-1, 1]^2$, called normally rising
homeomorphisms, and show that a normally rising homeomorphism  can have very complex
$\omega$-limit sets and $\alpha$-limt sets, though the homeomorphism itself looks very simple.
\end{abstract}


\maketitle

\pagestyle{myheadings} \markboth{J. Mai, E. Shi, K. Yan, and F. Zeng}{Limit Sets for Planar Homeomorphisms}

\section{Introduction}

By a {\it {dynamical system}}, we mean a pair $(X, f)$, where the {\it phase space} $X$ is a metric space and $f:X\rightarrow X$ is a continuous map. For $x, y\in X$, if there is a sequence of positive integers $n_1<n_2<\cdots$ such that $f^{n_i}(x)\rightarrow y$  then we call $y$ an {\it{{$\omega$-limit point}}}  of $x$.
We denote by  $\omega_f(x)$ or $\omega(x, f)$ the set of all $\omega$-limit points  of $x$ and call it the {\it $\omega$-limit set} of $x$.
The $\omega$-limit sets  are important in understanding the long term behavior of a dynamical system and the properties of which have been
intensively studied. It is well known that, if $X$ is a compact metric space, then $\omega_f(x)$ is nonempty, closed and strongly $f$-invariant for any $x\in X$. Bowen \cite{Bo75} gave an intrinsic characterization of abstract $\omega$-limit sets as those having no non-trivial filtrations and used shadowing to study the $\omega$-limit sets of Axiom A diffeomorphisms.  Hirsch, Smith, and Zhao \cite{HSZ01} showed that the $\omega$-limit set of any precompact orbit is internally chain transitive; the opposite direction is shown for tent maps with periodic critical points by Barwell-Davies-Good \cite {BDG12}, and for subshifts of finite type by Barwell-Good-Knight-Raines \cite{BGKR10}, respectively.  Barwell, Good, Oprocha, and Raines \cite{BGOR13} showed the equivalence between internal chain transitivity, weak incompressibility, and being an $\omega$-limit set for topologically hyperbolic systems. Good and  Meddaugh \cite{GM18} studied the relations between the collections of all $\omega$-limit sets and those of all internally chain transitive sets under various shadowing properties. The notion of $\omega$-limit set is also key in several definitions of attractors and chaos (see e.g. \cite{Mi85, Li93}).

\medskip
For one-dimensional systems, the structures of $\omega$-limit sets have been well characterized.
For interval maps $f:I\rightarrow I$ and $x\in I$, Blokh, Bruckner, Humke, and  Sm\'ital \cite{BBHS96} showed that
the family of all $\omega$-limit sets of $f$ forms a closed subset of the hyperspace  of $I$  endowed with the Hausdorff metric,
and Agronsky, Bruckner, Ceder,  and Pearson \cite{ABCP89} showed that $\omega_f(x)$ is either a finite periodic orbit, or an infinite nowhere-dense set, or the union of periodic nondegenerate subintervals of $I$. These results were extended to the cases when the phase space is
either a circle \cite{P02}, or a graph \cite{MS07, HM06, FHO22}, or a dendrite \cite{AEO07}, or a hereditarily locally connected continua \cite{S08}, or a quasi-graph \cite{MS17}.
However, when the phase space $X$ has dimension $\geq 2$, only partial results are known. For instance,  Agronsky and Ceder \cite{AC91} showed that every finite union of the nondegenerate Peano continua of the square $I^k$ is an $\omega$-limit set of some continuous map on the square $I^k$,
and Jim\'enez L\'opez and  Sm\'ital \cite{JS01} found the necessary and sufficient conditions for a finite union of Peano continua to be an $\omega$-limit set of a triangular map.  The family of all $\omega$-limit sets of the Stein-Ulam spiral map was  identified by Bara\'nski and Misiurewicz \cite{BM10, KM10}.
\medskip

Let $\mathbb Z$, $\mathbb Z_+$, and $\mathbb Z_-$ be the sets of integers, nonnegative integers, and nonpositive integers, respectively. Let $X$ be a topological space 
and $f:X\rightarrow X$ be a homeomorphism. Then the sets  $O(x, f)\equiv\{f^n(x):n\in\mathbb Z\}$, $O_+(x, f)\equiv\{f^n(x):n\in\mathbb Z_+\}$, and $O_-(x, f)\equiv\{f^{-n}(x):n\in\mathbb Z_-\}$ are called the {\it orbit}, {\it positive orbit}, and {\it negative orbit} of $x$, respectively.
For $x, y\in X$, if $y\in \omega_{f^{-1}}(x)$, then we call $y$ an {\it $\alpha$-limit point} of $x$. We denote by $\alpha_f(x)$ or $\alpha(x, f)$ the set of all $\alpha$-limit points of $x$ and call it the {\it $\alpha$-limit set} of $x$. 

\medskip

The aim of the paper is to study the $\omega$-limit sets and $\alpha$-limit sets of homeomorphisms on the plane $\mathbb R^2$. 
This topic has also been discussed by some ahtours. For instance, by means of $\omega$-limit sets and $\alpha$-limit sets,  Handel \cite{Ha99} obtained a fixed point theorem for homeomorphisms on the plane.
In Section 2, we introduce a class of homeomorphisms on the square $J^2$ where $J=[-1, 1]$, called {\it normally rising homeomorphisms}.
Every normally rising homeomorphism fixes point-wise the top edge and bottom edge of $J^2$, and moves up other horizontal line segments.
So, a normally rising homeomorphism looks very simple. However, the first main result we obtained shows that the $\omega$-limit set
and the $\alpha$-limit set of a normally rising homeomorphism can be very complex. Actually, any family of predescribed reasonable sets
can always be realised as the limit sets of a normally rising homeomorphism.
\medskip

For $s\in J$, let $J_s=J\times \{s\}$ and $\mathcal C_s$ be the collection of
all nonempty connected closed subsets of $J_s$. A map $\phi:J\rightarrow \mathcal C_s$ is {\it increasing}
if the abscissae and ordinates of the endpoints of $\phi(r)$ are increasing functions of $r$ and is {\it endpoint preserving}
if $\phi(-1)=(-1, s)$ and $\phi(1)=(1, s)$. If $f$ is
a normally rising homeomorphism on $J^2$ and $s\in J$, then $\omega_{sf}(r)\equiv\omega_f(r, s)$ defines an
increasing function $\omega_{sf}$ from $J$ to $\mathcal C_1$. Similarly, we also have an increasing function $\alpha_{sf}$ from
$J$ to $\mathcal C_{-1}$.
\medskip

Let $\mathcal A=\mathcal C_1$ and $\mathcal A'=\mathcal C_{-1}$. Then we have

\begin{thm}\label{main1}
Let \ \,$\mathbb N\,'$ and \ \,$\mathbb N\,''$\ \,be \,two
\,nonempty\, subsets\,\ of \ \,$\mathbb N\,$,\ \,and \;let \;\,${\mathcal V}\,=\,\{\,V_n\,:\,n\,
\in\,\mathbb N\,'\,\}$\ \ and \ \ ${\mathcal W}\,=\,\{\,W_j\,:\,j\,\in\,\mathbb N\,''\,\}$ \ be \,two\,
families \,of \;pairwise\ \,disjoint \,nonempty\ \,connected\ \,subsets\ \,of \;the\,
semi-open interval \ $(0, 1/2]$\;.\ \ For \,each \ $n\in\mathbb N\,'$\ \ and \,each \ $j\in\mathbb N\,''$\,,\ \ let \ \,$\mbox{\large$\omega$}_n\,:\,J\,\to\;\mathcal{A}$\ \; and \ \,$\mbox{\large$\alpha$}_j\,:\,J\, \to \;\mathcal{A}'$ \ be \,given \,increasing \,and\,\ endpoint \,preserving \,maps. \,Then \,there \,exists \,a \,normally \,rising \,homeomorphism \ $f\,:\,J^{\,2}\,\to\;J^{\,2}$\ \,such\ \,that \ $\mbox{\large$\omega$}_{sf\,}=\ \mbox{\large$\omega$}_n$\ \ for \,any \ $n \in\mathbb N\,'$\ \ and \,any \ $s\in V_n$\ , \ and \ \,$\mbox{\large$\alpha$}_{tf\,}=\ \mbox{\large$\alpha$}_j$ \ for
\,any \ $j \in\mathbb N\,''$\ \ and \,any \ $t\in W_j$.
\end{thm}

Let $d$ be the Euclidean metric on $\mathbb R^2$. Let $f:\mathbb R^2\rightarrow \mathbb R^2$ be a homeomorphism and $x\in\mathbb R^2$. The orbit
$O(x, f)$ is said to be  {\it{positively divergent}} (resp. {\it{negatively divergent}}) if $d(f^n(x), (0, 0))\rightarrow \infty$ (resp. $d(f^{-n}(x), (0,0))\rightarrow \infty$) as $n\rightarrow \infty$. If $O(x, f)$ is both positively divergent and negatively divergent,
then we say $O(x, f)$ is {\it doubly divergent}.

\medskip
A subset $E$ of $\mathbb R^{\,2}$ is a {\it disk} if it is homeomorphic to the unit closed
ball of $\mathbb R^{\,2}$. We use $\partial E$ and $\stackrel{\circ}{E}$ (or, ${\rm Int}(E)$) to denote the boundary
and interior of $E$, respectively. By means of Theorem \ref{main1},  we get in Section 3 the following somewhat surprising result.

\begin{thm}
Let \ $E$\ \ be a disk in \ $\mathbb R^{\,2}$\,.\,\
Then there exists a homeomorphism \ $h:\mathbb R^{\,2}\to\mathbb R^{\,2}$\ \;such that\,, \ for any
\ $x\in\,\partial E$\,,\ \ the \,orbit \ $O(x,h)$\ \;is \,bounded\;, \,but \,for any\ \ $y
\in\, \stackrel{\circ}{E}$\,,\ \ the \,orbit \ $O(y,h)$ \;is \,doubly\,-\;\!divergent.
\end{thm}

\section{Limit sets of normally rising homeomorphisms on $J^2$}
In this paper\,,\ \,for\ \ $r\,,\,s\,\in\,\mathbb R$\,,\ \,we use \ $(r,s)$ \ to
denote a point in \,$\mathbb R^{\,2}$\,.\ \ If \ $r\,<\,s$\;,\ \ we also use \ $(r,s)$\ \ to
denote an open interval in\ \,$\mathbb R$.\,\,\,These will not lead to confusion\,.\,\,\,For
example\,, \,if we write \ $(r,s)\,\in\,X$\,, \ then \ $(r,s)$ \ will be a point\,; \
if we write \ $t\,\in\,(r,s)$\,,\ \ then \ $(r,s)$ \ will be a set\,,\ \,and hence is
an open interval\,.
\medskip

We always write\ \ $J=\,[\,-1\,,\,1\,]$\;.\ \ \,Define the homeomorphism
\ $f_{01}\,:\,J\,\to\;J$\ \ by\ , \ \,for any \ $s\,\in\,J$,
$$
f_{01}(s)\,=\;
   \left\{\begin{array}{ll}(s+1)/2\,, & \mbox{\ \ \ \ \ if \ }\;0\,\le\,s\,\le1\;; \\
                        \!\;\;s+1/2\,, & \mbox{\ \ \ \ \ if \ \ }-1/2\,\le\,s\le\,0\;;  \\
                             \;2s+1\,, & \mbox{\ \ \ \ \ if \ }-1\,\le\,s\le\,-1/2\ .
\end{array}\right.
$$
Define the homeomorphism \ $f_{02}\,:\,J^{\,2}\,\to\;J^{\,2}$\ \ by\ , \ \,for any \
$(r,s)\,\in\,J^{\,2}$\ ,
$$
f_{02}(r,s)\ =\ \big(\,r\,,\,f_{01}(s)\,\big)\ .
$$

For any compact connected manifold $M$,\ \ denote by \,$\partial M$ \,the boundary\,,
\,and by $\stackrel{\ \circ}{M}$ \,the interior of \,$M$. \ Specially, \ we have \ $\partial J\,
=\,\{\,-1\,,\,1\:\!\}$\,, \ and \ $\stackrel{\ \circ}{J}\,=\,(\,-1\,,1\:\!)$\,. \ Note that
\ $\partial J^{\,2}$ \ is \ $\partial(J^{\,2})$\,,\ \ not \ $(\partial J)^{\,2}$\,.

\begin{defin}\label{normal rising} \ \ For any\ \ $s\,\in\,J$\,,\ \ \,write \ $J_s\,=
\,J\times\{s\}$\;. \ A homeomorphism \ $f:J^{\,2}\to\,J^{\,2}$ \ is said to be \,{\it
normally rising} \ if
$$
f\,|\;\partial J^{\,2}\,=\,f_{02}\,|\;\partial J^{\,2}\,, \hspace{10mm}    \mbox{and}
\hspace{8mm} f(J_s)\;=\,f_{02}(J_s) \hspace{5mm} \mbox{for \ any}\ \ \,s\,\in\,J.
$$
\end{defin}

 Define the maps \ $p\;,\;q\;:\,\Bbb{R}^{\,2}\to\,\Bbb{R}$ \ by \ $p(x)\,=\,r$
\ and \ $q(x)\,=\,s$\ \ for any \ $x\,=\,(r,s)\,\in\,\Bbb{R}^{\,2}$\,,\ \;that is,\ \
we denote by \ $p(x)$\ \ and \ $q(x)$\ \ the abscissa and the ordinate of \,$x$\,,\,\
respectively\,. \,By the definition\,, \ the following lemma is obvious\,.

\begin{lem}\label{rising-prop}Let \ $f:J^{\,2}\to J^{\,2}$ \ be a normally rising
homeomorphism\,. \;Then

\vspace{1mm}(1)\ \ \ $f\,|\,{J_{-1}\cup J_1}$ \ is the identity map\ ;

\vspace{1mm}(2)\ \ \ $pf(r_1\,,\,s)\,<\,pf(r_2\,,\,s)$\;,\ \ for any \ $\{r_1\,,\,r_2
\,,\,s\}\,\subset\,J$ \ with \ $r_1\,<\,r_2$ ;

\vspace{1mm}(3)\ \ \ $\omega_f(x)\,=\,\alpha_f(x)\,=\,\{x\}$\;, \ \ for any \ $x\,\in
\,J\times\,\partial J$ ;

\vspace{1mm}(4)\ \ \ $\omega_f(r,s)=(r,1)$\;,\ \ \,and \ \,$\alpha_f(r,s)=(r,-1)$\;,\ \ \ for
any \ $(r,s)\,\in\;\partial J\times\stackrel{\ \circ}{J}$\;;

\vspace{1mm}(5)\ \ \ for any \ $(r,s)\in\ \stackrel{\ \circ}{J}\,^{\,2}=(-1, 1)^{\,2}$\,, \ \
$\omega_f(r,s)$ \ is a nonempty connected closed subset of $J_1$\ ,\ \ and $\alpha_f(r,s)$ is a
nonempty connected closed subset of \,$J_{-1}$\ .
\end{lem}

 Denote by ${\mathcal A}$ \;the family of all nonempty connected closed subsets
of $J_1$ , \ and by \,$\mathcal A'$ \;the family of all nonempty connected closed subsets
of $J_{-1}$ . \ From Lemma \ref{rising-prop} we see that \
$$
\omega_f(x)\,\in\,{\mathcal A}\ ,\ \ \ \ \ \mbox{ and}\ \ \ \ \ \ \alpha_f(x)\,\in\,{\mathcal A'}\ ,\
\ \ \ \mbox{for \ any}\ \ \ x\,\in\,J\times\stackrel{\ \circ}{J}\ .
$$

\begin{defin}\label{increasing} Let \ $v_1=(\,-1\,,\;\!1)\ ,\ \ v_2=(\,1\,,\;\!
1)\ , \ v_3=(\,-1\,,\,-1)$\ \ \,and \ \,$v_4=(\,1,-1)$\ \ be the four vertices of the
square \,$J^{\,2}$\,. \,\ A map \ \,\mbox{\large$\omega$}$\;:\,J\,\rightarrow\;\mathcal{A}$ \ is said
to be \;{\it increasing} \;if \
$$
\min\big(\,p\;\mbox{\large$\omega$}(r_1)\,\big)\,\leq\;\min\big(\,p\;\mbox{\large$\omega$}(r_2
)\,\big)\ \ \ \ \ \ \mbox{and}\ \ \ \ \ \ \max\big(\,p\;\mbox{\large$\omega$}(r_1)\,\big)
\,\leq\;\max\big(\,p\;\mbox{\large$\omega$}(r_2)\,\big)
$$
for any $\{r_1\,,\,r_2\}\in J$ \ with \ $r_1<r_2$\;. \ The map \ $\mbox{\large$\omega$}$\
\ is said to be \;{\it endpoints preserving} \;if \ $\mbox{\large$\omega$}(-1)=\{v_1\}$ \
and \ $\mbox{\large$\omega$}(1)=\{v_2\}$\,. \

\vspace{1mm}Similarly\,, \,a map \ $\mbox{\large$\alpha$}\;:\,J\,\rightarrow\;\mathcal{A}'$ \ is said
to be \,{\it increasing} \;if \
$$
\min\big(\,p\;\mbox{\large$\alpha$}(r_1)\,\big)\,\leq\;\min\big(\,p\;\mbox{\large$\alpha$}(r_2
)\,\big)\ \ \ \ \ \ \mbox{and}\ \ \ \ \ \ \max\big(\,p\;\mbox{\large$\alpha$}(r_1)\,\big)
\,\leq\;\max\big(\,p\;\mbox{\large$\alpha$}(r_2)\,\big)
$$
for any $\{r_1\,,\,r_2\}\in J$ \ with \ $r_1<r_2$\;. \ The map \ $\mbox{\large$\alpha$}$\
\ is said to be \;{\it endpoints preserving} \;if \ $\mbox{\large$\alpha$}(-1)=\{v_3\}$ \
and \ $\mbox{\large$\alpha$}(1)=\{v_4\}$\,.
\end{defin}

From \,Lemma \ref{rising-prop} \,we get the following lemma at once.

\begin{lem}\label{rising-prop2}Let \,$f\,:\,J^{\,2}\rightarrow\,J^{\,2}$\,\,\,be a\,\
normally \,rising \,homeomorphism\,.\ \ \vspace{1mm} For\ \,any \,given\ \ $s\in\stackrel{\
\circ}{J}$\,, \ define\; maps \ \,$\mbox{\large$\omega$}_{sf}\;:\,J\,\rightarrow\;\mathcal{A}$ \ and\
\ \,$\mbox{\large$\alpha$}_{sf}\;:\,J\,\rightarrow\;\mathcal{A}'$ \ by
\begin{equation}\label{eq:2.1}
 {\large\omega}_{sf}(r)\;=\ \omega_f(r,s) \ \ \ \ and \ \ \ \
 {\large\alpha}_{sf}(r)\;=\ \alpha_f(r,s), \ \ \
for\ \ any \ \ r\,\in\,J\;.
\end{equation}
Then \,both \ \,$\mbox{\large$\omega$}_{sf}$ \ and \ \,$\mbox{\large$\alpha$}
_{sf}$\ \ are\ \,increasing \,and \;endpoints \,preserving\;.
\end{lem}

Write\ \ \ \ \ $t_0\,=\ 0\ ,\ \ \ \ I_1\,=\;\big(\; 0\:,\;\!1/\;\!2\;\big
]\ ,\ \ \ \ D_1\,=\;J\times\big[\;0\:,\;\!1/\;\!2\;\big]$\ ,\ \ \ \ \ \ and \vspace{-0mm}
$$
t_n\,=\;f_{\;\!01}^{\,n}(t_0)\ ,\ \ \ \ \ D_n\,=\;f_{\;\!02}^{\,n-1}(D_1)\ ,\ \ \ \ \
\mbox{for\ \ any}\ \ \ n\;\in\ \mathbb{Z}\ .
$$
Then \ $D_n\,=\;J\times[\,t_{n-1}\,,\,t_n\,]$ , \ and \ $D_n\,\cap\,D_{n+1}\,=\;J_{t_n}$\ .\ \ A main\ result of this\ paper is the following theorem\,, which shows that the \;$\omega$-limit\ \,sets \;and \;$\alpha$-limit sets of homeomorphisms of\ \,$J^{\,2}$\ \,may have
very complex structure\,,\ \,even if these homeomorphisms are normally rising.

\begin{thm}\label{main1'}
Let \ \,$\mathbb N\,'$ and \ \,$\mathbb N\,''$\ \,be \,two
\,nonempty\, subsets\,\ of \ \,$\mathbb N\,$,\ \,and \;let \;\,${\mathcal V}\,=\,\{\,V_n\,:\,n\,
\in\,\mathbb N\,'\,\}$\ \ and \ \ ${\mathcal W}\,=\,\{\,W_j\,:\,j\,\in\,\mathbb N\,''\,\}$ \ be \,two\,
families \,of \;pairwise\ \,disjoint \,nonempty\ \,connected\ \,subsets\ \,of \;the\,
semi-open interval \ $I_1$\;.\ \ For \,each \ $n\in\mathbb N\,'$\ \ and \,each \ $j\in\mathbb N\,''$\,,\ \ let \ \,$\mbox{\large$\omega$}_n\,:\,J\,\to\;\mathcal{A}$\ \; and \ \,$\mbox{\large$\alpha$}_j\,:\,J\, \to \;\mathcal{A}'$ \ be\, arbitarily\,given \,increasing \,and\,\ endpoint \,preserving \,maps. \,Then \,there \,exists \,a \,normally \,rising \,homeomorphism \ $f\,:\,J^{\,2}\,\to\;J^{\,2}$\ \,such\ \,that \ $\mbox{\large$\omega$}_{sf\,}=\ \mbox{\large$\omega$}_n$\ \ for \,any \ $n \in\mathbb N\,'$\ \ and \,any \ $s\in V_n$\ , \ and \ \,$\mbox{\large$\alpha$}_{tf\,}=\ \mbox{\large$\alpha$}_j$ \ for
\,any \ $j \in\mathbb N\,''$\ \ and \,any \ $t\in W_j$, where \,the \,definitions\, of \,\ $\mbox{\large$\omega$}_{sf}$\ \ and \ \,$\mbox
{\large$\alpha$}_{tf}$ \ are \,given \,by \;$(\ref{eq:2.1})$\,.
\end{thm}

\begin{rem} \ Before  the  proof of  Theorem \ref{main1'}\,,\ \,we make a
survey of the differences between \ $\omega\ ,\ \omega_f\ ,\ \mbox{\large$\omega$}_{sf}$ \ and \ $
\mbox{\large$\omega$}_n$\ .\ \ Let\ \ ${\mathcal F}$\ \ be the family  of all normally rising
homeomorphisms from \,$J^{\,2}$ \,to \,$J^{\,2}$\,.\ \ If we consider only the space\
\,$J^{\,2}$ \,and the family \ ${\mathcal F}$\ , \ then \

(1) \ $\omega$ \ is a map from \,$J^{\,2}\times{\mathcal F}$ \,to the family \ $\mathcal A$ \ of some
subsets of $J_1$\;; \

(2) \ $\omega_f$ \ is a map from \,$J^{\,2}$ \,to \,$\mathcal A$ , with a given \ $f\in{\mathcal
F}$ \ as a parameter ; \

(3) \ $\mbox{\large$\omega$}_{sf}$ \ is a map from \,$J$ \,to \ $\mathcal A$ , \ with a given
\ $s\in J$ \ and an \ $f\in{\mathcal F}$ \ as parameters\;; \

(4) \ $\mbox{\large$\omega$}_n$ \ is a map from \,$J$ \,to \,$\mathcal A$ , \ with a given \
$n\in\mathbb N$ \ as a parameter ; \

(5) \ The definition of\ \,$\omega$\ \,depends on the definition of\ \,$\omega$-limit sets of
orbits of maps\,, \,and \,$\omega_f$ \,can be regard as the restriction of \ $\omega$ \ to\,\
$J^{\,2}\times\{f\}$ ( really\,, \,it should be \ $\omega_f=\,\omega\,{\bf i}_f$,\ \ that is the
composition of an imbedding \ ${\bf i}_f:J^{\,2}\rightarrow J^{\,2}\times{\mathcal F}$ \ and \ $\omega$ )
.\ \ $\mbox{\large$\omega$}_{sf}$ \ can be again regard as the restriction of \ $\omega_f$\ \
to \ $J_s$ ( really\,, \,it should be \ $\mbox{\large$\omega$}_{sf}=\,\omega_f\,{\bf i}_{\:\!
s}$\ ,\ \ that is the composition of an imbedding \ ${\bf i}_{\:\!s}:J\rightarrow J^{\,2}$\ \
and \ $\omega_f$ )\ .\ \,Hence\,,\ \,the definitions of\ \ $\omega \ ,\ \omega_f$\ \ and $\mbox{\
\large$\omega$}_{sf}$ \ depend on orbits of maps\;. \ However\,, \ $\mbox{\large$\omega$}_n$\
\,is only a pure map , \,of which the definition is directly assigning a set $\mbox{\
\large$\omega$}_n(r)\,\in\mathcal A$\ \ for each\ \ $r\in J$\ ,\ \ which  does not depend on
any \ $f\in\,\mathcal F$ .

\vspace{1mm}The differences between \ $\alpha\ ,\ \alpha_f\ ,\ \mbox{\large$\alpha$}_{sf}$ \ and \
$\mbox{\large$\alpha$}_n$\ \ are analogous\,.

\end{rem}

Now we begin the proof.\ \ For $k\in\mathbb N$, write $\mathbb N_k=\{1,2,\cdots,k\}$.

\begin{proof}
We may consider only the case that \ $\mathbb N
\,'\,=\;\mathbb N$\ \ since the case that\ \ $\mathbb N\,'$\ \;is a finite  subset of\ \ $\mathbb N$\ \ is
similar and is simpler.
\medskip

If \ \,$1/2\,\notin\,\bigcup_{\,n=1}^{\;\infty}V_n$\ ,\ \ then we can add
the one-point-set \ $\{1/2\}$ \ to the family \,$\mathcal{V}$\;.\ \ If \,there exist \;$n
\in\mathbb N$\ \ and \;$s\in(\:0\:,\,1/2\:)$\ \ such that \ $[\:s\:,\,1/2\:]\,\subset\,V_n$\;,\ \
then we can divide \,$V_n$ \;into two connected sets \;$\{1/2\}$\ \;and \;$V_n\;-\,\{1/2\}$ .\ \;Therefore, \;we may assume that \;$\{1/2\}\,\in\,\mathcal{V}$\;, \ and \ $V_1\;\!=\,\{1/2\}$\;.
\medskip

For any \ $n\;,\;k\,\in\,\mathbb N$\ , \ write                      \vspace{2mm}
$$
V_{nk}\;=\;\left\{\begin{array}{ll} V_n\ ,
$$
&\vspace{-0.3mm}\mbox{\ \ \ if \ } V_n  \mbox{\ \ is \,a \,one-point-set}          \\
& \vspace{2.7mm}\hspace{23.5mm}              \mbox{ or \,a \,closed \,interval}\ ; \\
                             \big[\,a\;,\,(a+k\:\!b)/(k+1)\,\big]\ ,
& \vspace{2mm} \mbox{\ \ \ if \ } V_n\; =\;[\,a\;,\,b\,) \mbox{ \ for some \ }a\,<\,b
\ ;   \\                     \big[\,(k\:\!a+b)/(k+1)\;,\;b\,\big]\ ,
& \vspace{2mm} \mbox{\ \ \ if \ } V_n\;=\;(\,a\;,\,b\,] \mbox{ \ for some \ } a\,<\,b
\ ;   \\                     \big[\,(k\:\!a+b)/(k+1)\;,\,(a+k\:\!b)/(k+1)\,\big]\ ,
& \vspace{0mm} \mbox{\ \ \ if \ }V_n\;=\;(\,a\;,\,b\,)\mbox{ \ for some \ }a\,<\,b\;.
\end{array}\right.                                                       \vspace{3mm}
$$
Then \ $V_{n1}\,\subset\,V_{n2}\,\subset\,V_{n3}\,\subset\,\cdots\,\subset\,V_n$\ ,\ \ \ $\bigcup_{\,k\,=
\,1}^{\;\infty}V_{nk}\,=\,V_n$\ ,\ \ \ and \,for \,every \,fixed\ \ $k\,\in\,\mathbb N$\;, \
$\{\:\!V_{nk}\:\!:\;\!n\;\!\in\;\!\mathbb N\,\}$ \;\!is \;\!a \;\!family \;\!of\ \,pairwise\
\,disjoint\,\ nonempty \,connected \,closed \,subsets \,of \ $I_1$\;.

\medskip
For each\ \,$n\,\in\,\mathbb N$\;,\ \ since\ \ $\mbox{\large$\omega$}_n\,:\,J\,\rightarrow\;
\mathcal{A}$ \ is increasing and endpoint preserving , \,there exist increasing functions
\ $\xi_{\,n1}\,:\,J\,\rightarrow\;J$ \ and \ $\xi_{\,n2}\,:\,J\,\rightarrow\;J$ \ such that\,,\ \,for
any \ $r\in J$\ , \ it holds that
\begin{equation}\label{eq:2.2}
\xi_{\,n1}(r)\,\leq\,\xi_{\,n2}(r)\ , \hspace{8mm} {\large\omega}_n(
r)\,=\;\big[\;\xi_{\,n1}(r)\;,\;\xi_{\,n2}(r)\;\big]\,\times\,\{1\}\ ,
\end{equation}
and \,\ $\xi_{\,n1}(-1)\,=\,\xi_{\,n2}(-1)\,=\,-1$ , \,\ $\xi_{\,n1}(
1)\,=\,\xi_{\,n2}(1)\,=\,1$ . \ For \,$j\,=\,1\,,\,2$ , \ let \ $Y_{nj}$ \ be the set
of all discontinuous  points of\ \,$\xi_{\,nj}$\ .\ \ Since  the set of discontinuous
points of an increasing function is countable\,,\ \,there  exists a  countable  dense
subset \ $R\,=\,\{\,r_{-\,1}\;,\,r_0\;,\,r_1\:,\,r_2\:,\,r_3\:,\,\cdots\,\}$ \ of \,$
J$\ \,such that \;\,$r_{-\,1}=\,-1\;,\ \ r_0\,=\,1$\;, \,\ and \,\ $\bigcup_{\,n\,=\,
1\,}^{\;\infty}(\,Y_{n1}\,\cup\,Y_{n2\,})\,\subset\,R$\;.
\medskip

For each \ $k\;\in\;\mathbb{N}$ ,\ \ \ write                   \vspace{-2mm}
$$
R_k\;=\;\,\{\,r_{-\,1}\;,\,r_0\;,\,r_1\;,\,r_2\;,\,\cdots\;,\,r_k\,\}\ , \hspace{8mm}
\mbox{and\ \hspace{6mm} $V\;=\;\bigcup_{\,n\,=\;\!1}^{\;\infty}V_n$}\ .     \vspace{-2mm}
$$
Then we have \begin{equation}                                                           \vspace{2mm}
  \bigcup_{\,k\,=\;\!1}^{\;\infty}\big(\:R_k\times(\,V_{1k}\,\cup\,V_{2k}\,\cup
\,\cdots\,\cup\,V_{kk}\,)\,\big)\ =\,\ R\,\times V\ \subset\;\,D_1\,-\,J_0\;\,.
\end{equation}
For\ \ any\ \ $j\;,\;m\;\in\;\mathbb{Z}$\ \ \,with \ $j\;\leq\;m$\ ,\ \ write
$$
\mbox{$D_j^{\,m}\;=\;\,\bigcup_{\;i\,=\,j\,}^{\;m}\,D_i\ \,,\ \ \ \ \ D_j^{\,\infty}\;=\;
\,\bigcup_{\;i\,=\,j\,}^{\;\infty}\,D_i\ \,,\ \ \ \ \mbox{and}\ \ \ \ \ D_{-\,\infty}^{\;m}\;
=\;\,\bigcup_{\;i\,=\,-\,\infty\,}^{\;m}\,D_i\ \,.$}                         \vspace{1mm}
$$
Then \ $D_j^{\,m}\,=\ J\times\;\![\;t_{j-1}\,,\,t_m\,]$ , \,\ $D_j^{\,\infty}\,=\ J\times\;\![\;t
_{j-1}\,,\,1\,)$ , \,\ \vspace{1mm} $D_{-\,\infty}^{\;m}\,=\ J\times(\;-1\,,\,t_m\,]$\ ,\ \,\
and \ the \ closure \ \ \,$\overline{\;\!\!D_j^{\,\infty}\!}\ \,=\;\,D_j^{\,\infty}\,\cup\;J_1$\ ,\
\ \ $\overline{\;\!\!D_{-\,\infty}^{\;m}\;\!\!\!}\ \,=\ D_{-\,\infty}^{\;m}\,\cup\;J_{-\,1}$\ \,.
\ (See Figure 2.1).

\vspace{3mm}
In order to construct the homeomorphism $f:J^{\,2} \to J^{\,2}$ mentioned
in Theorem \ref{main1'}\,, \;we                                                   \vspace{3mm}

\ {\bf First},\ \ \ \ put \ \ \ $f_0\ =\ f\ |\;D_0\ =\;\,f_{02}\ |\ D_0\ :\ D_0\ \to\;
\,D_1$ .
\medskip

{\bf Secondly} ,\ \ assume  that ,\ \ for  some \ $k\in\mathbb N$\ ,\ \ we  have
defined a homeomorphism
\begin{equation}\label{eq:2.4}
f_{\ref{eq:2.4}}\;=\ f\;|\,D_0^{\,k^{2}-\,1}\ :\ \ D_0^{\,k^{2}-\,
1}\;\,\to\ \ D_1^{\,k^{2}},
\end{equation}
which satisfy the following conditions : ( For  avoiding  that  there
are overmany subscripts\,, \,in the already explicit domain \,$D_0^{\,k^{2}-\,1}$ , \
we will use \,$f$ \,to replace \,$f_{\ref{eq:2.4}}$\;, \,although the entire \,$f:J_{\,2}\to\,
J_{\,2}$ \ has not yet been defined\,. \,Similarly here in after. )

\vspace{3mm}(C.1.$k$)\ \ \ $f(J_s)\;=\,f_{02}(J_s)$\ ,\ \ \ \ for \ any \ \,$J_s\;\subset\
D_0^{\,k^{2}-\,1}$ ;

\vspace{2mm}(C.2.$k$)\ \ \;\,$|\,pf(r,s)-r\,|\,<\,2/(2m+3)$\ ,\ \ \,for any \ $m\,\in
\,\mathbb N_{\,k}$ \ and any\ \ $(r,s)\,\in\,D_{(m-1)^{\:\!2}}^{\,m^{2}-1}$\ ;

\vspace{2mm}(C.3.$k$)\ \ \;\,$|\,pf(r,s)-r\,|\,<\,2/(2k+5)$\ , \ \ for any \ $(r,s)\,
\in\,D_{k^2-1}\,\cap\;D_{k^2}\,=\,J_{\:\!t_{\:\!k^2-1}}$ .
\medskip

We will extend the homeomorphism \,$f_{\ref{eq:2.4}}$ \,to a homeomorphism
\begin{equation}\label{eq:2.5}
f_{\ref{eq:2.5}}\;=\ f\;|\;D_0^{\,(k+1)^{\,2}-1}\ :\ \ D_0^{\,(k+1
)^{\,2}-1}\;\,\to\ \ D_1^{\,(k+1)^{\,2}}
\end{equation}
 as follows :

\medskip
{\bf Step 1.} \ \ For any \ \,$n\,\in\;\mathbb N_{\,k}$ \ \,and any \ \,$(r,s)\,
\in\,R_k\times V_{nk}$ , \ we define \vspace{0.5mm} \ $f\big(\;\!f^{\,i\,-1}f^{\,k^2-1\,}
(r,s)\;\!\big)$\ $=\;f^{\,i}\big(f^{\,k^2-1\,}(r,s)\;\!\big)$\ \ for\ \ $i\,=\,1\,,\,
\cdots\,,\,2k+1$ \ in the natural order by putting the ordinate
\begin{equation}\label{eq:2.6}
q\big(f\big(\;\!f^{\,i\,-1}f^{\,k^2-1\,}(r,s)\;\!\big)\big)
\,=\;\,q\big(\;\!f^{\,i}\big(\;\!f^{\,k^2 - 1}(r, s)\,\big)\;\!\big)\,
=\;\,q\big(\;\!f_{\;\!02}^{\,k^{2}+\,i\,-1}(r,s)\;\!\big),
\end{equation}
and putting the abscissa
\begin{equation}\label{eq:2.7}
p\;\!\big(f\;\!\big(\;\!f^{\,i\,-1}f^{\,k^2-1\,}(r,s)\;\!
\big)\;\!\big)\,=\;\,p\;\!\big(\;\!f^{\,i}\big(\;\!f^{\,k^2-1}(r, s)\,\big)\;\!\big)
\end{equation}
$$=\ p\;\!\big(\;\!f^{\,k^2-1}(r,s)\;\!\big)\,+\;i\,\big(\;
\xi_{\,n\;\!\lambda_k}(r)\; - \;p\;\!\big(\,f^{\,k^2-1}(r , s)\,\big)\,\big)\;\!\mbox{\bf
\large$/$}\:\!\big(2k+7\big),$$
where \ $\lambda_k\,=\,1$ \ \,if \,$k$ \,is odd\ , \ \,and \ $\lambda_k\,=\,2
$ \ \,if \,$k$ \,is even .\ Note that\;, \;for any \ $n\,\in\;\mathbb N_{\,k}$ \ and any
\,$(r,s)\,,\,(r\,'\!,s)\in R_k\times V_{nk}$ \ with \;$r<r\,'$, \;since
\ $\mbox{\large$\omega$}_n$\ \,is \,increasing\,, \,from \,(\ref{eq:2.7}) \,we get
$$p\;\!\big(\;\!f^{\,i}\big(\;\!f^{\,k^2-1}(r,s)\,\big)\;\!
\big)\,<\,p\;\!\big(\;\!f^{\,i}\big(\;\!f^{\,k^2-1}(r\,'\!,s)\,\big)\;\!\big)\ \ \ \
\ \,\mbox{for}\ \ i\,=\,1\,,\,\cdots\,,\,2k+1\,.
$$

\medskip{\bf Step 2.} \ \ In Step 1\,, \ for \ $i\,=\,1\,,\,\cdots\,,\,2k+1$ ,\ \
the\ \,set\ \ $f^{\,i\,-1}\big(\,f^{\,k^2-1}\big(\,R_k\,\times\,\bigcup_{\,n\;=\,1}^{\;k}
V_{nk}\,\big)\,\big)$  $\ \subset\ D_{k^2+i-1}\,-\;J_{\,t_{\,k^2+i-2}}$\ \ has been really
defined\,. \,Thus we can define
\begin{equation}\label{eq:2.8}
X_k\ =\ \bigcup_{\;i\;=\,1}^{\;2k+1}\,f^{\,i\,-1}\big(\,f^
{\,k^2-1}\big(\,R_k\,\times\,\bigcup_{\,n\;=\,1}^{\;k}V_{nk}\,\big)\,\big) .
\end{equation}
Note that \ $X_k\,\subset\,\bigcup_{\;i\;=\,1}^{\;2k+1}\big(\,D_{k^2+i-1}
\,-\;J_{\,t_{\,k^2+i-2}}\,\big)\; = \;D_{k^2}^{k^2+2k}\,-\;J_{\,t_{\,k^2-1}}$\ .\ \,\
Write \vspace{1mm} \ \,$S_k\,=\;q(X_k)$\;.\ \ \,Then \ \,$S_k\,=\;\{\,s\,\in\,J\,:\;J
_s\,\cap\,X_k\;\neq\;\,\emptyset\,\}$\ .\ \ Noting that \ $R_k$\ \ contains just\ \ $k+2$\ \
points\,, \ \,$R_k\times\bigcup_{\,n\;=\,1}^{\;k}V_{nk}$ \ \,contains just \,$k\;\!(k+2)$
\,connected components\,,\ \ and every connected component of \ $R_k\times\bigcup_{\,n\;=
\,1}^{\;k}V_{nk}$\ \ is a point or a vertical line segment in\ \ $D_1-J_0$\,,\ \,from
\,(\ref{eq:2.6})\;, \,(\ref{eq:2.7})\;, \ (\ref{eq:2.8}) \ and the conditions \,(C.1.$k$)\ \,we see that \,$X_k$
\,with \,$S_k$ \ has the following properties :

\medskip
(A)\ \ $X_k$\ \  contains just\ \ $k(k+2)(2k+1)$\ \
connected components\,, \ \,$S_k$ \ contains just \ $k(2k+1)$\ \ connected components
,\ \ and  every  connected  component of\ \ $X_k$\ \ is a point or an arc\,,\ \ every
connected component of \,$S_k$\ \ is a point or a closed interval ;
\medskip

(B)\ \ For any\ \,$s\,\in\,S_k$\ , \ \,$J_s\,\cap\,
X_k$\ \ contains just\ \ $k+2$ \ points\,. \ Specially\,,\ \ since \ $\{\,r_{-1}\,,\,
r_0\}\,=\,\{\,-1\,,\,1\}\,\subset\,R_k$\;, \ \ $(\,-1,s)$ \ and \ $(1,s)$ \ are two points
in \ $J_s\,\cap\,X_k$ ;
\medskip

(C)\ \ If \;$L$ \;is a connected component of \,$S_k$\ \ and \;is a closed interval\ ,\ \,then \ $(J\times L)\,\cap\,X_k$\ \ is the union of
just \ $k+2$ \ arcs\;,\ \,of which each arc is a connected component \,of \,$X_k$\;.\
\ Specially\,, \ $\{-1\}\times L$ \ and \ $\{1\}\times L$ \ are two connected components \,of
\,$X_k$ . \ Moreover\,, \ for any \ $s\in L$ \ and any connected component \,$A$ \,of
\,$X_k$\ \,in\ \ $(J\times L)\,\cap\,X_k$\;,\ \ $A\,\cap\,(J\times\{s\})$ \ contains just one
point\,.
\medskip

In Step 1\,, \,we actually have given the definition of \ $f\,|\,X_k$\ ,\
hence we actually have obtained the map
\begin{equation}\label{eq:2.9}
f_{\ref{eq:2.9}}\;=\ f\;|\;D_0^{\,k^2-1}\cup X_k\ :\ \,D_0^{\,k^2-
1}\cup X_k\;\,\to\,\,\ D_1^{\,(k+1)^2}\ .
\end{equation}
From \,(\ref{eq:2.6}) \,and \,(\ref{eq:2.7}) \,we see that the map\ \,$f_{\ref{eq:2.9}}$\ \,is a
continuous injection\,.\ \ From the properties of \,$X$ \,mentioned above we see that
the map \,$f_{\ref{eq:2.9}}$ \,can be uniquely extended to a continuous injection
\begin{equation}{\label{eq:2.10}}
f_{\ref{eq:2.10}}\;=\ f\;|\;D_0^{\,k^2-1}\,\cup\;(\,J\times S_k\,)\ :\
\,D_0^{\,k^2-1}\,\cup\;(\,J\times S_k\,)\,\;\to\ \ D_1^{\,(k+1)^{\,2}}
\end{equation}
such that\,,\ \ for any \,$s\,\in\,S_k$ \ and any connected component
\;$L$\ \;of \ $J_s\,-\;X_k$\ ,\ \ \,$f\;|\;L$\ \ is linear\,.\ \ Such an injection\,\
$f_{\ref{eq:2.10}}$\ \,will be called the \,{\it level linear extension} \,of \,$f_{\ref{eq:2.9}}$\;. \
Obviously\,, \,$f_{\ref{eq:2.10}}$ \,can also be uniquely extended to a homeomorphism
$$
f_{\ref{eq:2.5}}\;=\ f\;|\;D_0^{\,(k+1)^{\,2}-1}\ :\ \ D_0^{\,(k+1)^{\,2}-1}\;\,\to\ \ D_1^{\,(k+1)^{\,2}}
$$
such that\,,\ \,for any\ \ $r\,\in\,J$\ \ and any connected component
\,$L$\ \,of\ \ $\big(\,\{r\}\times J\,\big)\,\cap\;D_{k^2}^{\,(k+1)^2-1}\,-\,(\;\!J\times S_k
\:\!)$\ ,\ \ \ $f\;|\;L$ \ is linear\,. \ Such a homeomorphism \,$f_{\ref{eq:2.5}}$\ \,will be
called the {\it vertical linear extension} \,of the injection \,$f_{\ref{eq:2.10}}$ .
\medskip

Since the level linear extension  from \,$f_{\ref{eq:2.9}}$ \,to \,$f_{\ref{eq:2.10}}$ \,is
before the vertical linear  extension from \,$f_{\ref{eq:2.10}}$ \,to\ \,$f_{\ref{eq:2.5}}$\ ,\ \ by\,\
(\ref{eq:2.6})\ \,and \,(C.1.$k$)\ \,we see that\,,\ \ for the homeomorphism \,$f_{\ref{eq:2.5}}$\,,\ \
the condition\,\ (C\,.\,1\,.\,$k+1$) \,holds\,.
\medskip

Since \ $\mbox{\large$\omega$}_n$ \ is \,endpoint preserving , \ in \,(\ref{eq:2.7}) ,
\  \vspace{1.5mm}  if \ $r\,\in\;\partial J$ \ then \ $p\;\!\big(\,f^{\,k^2-1}(r,s)\,\big)
\;=\;\xi_{\,n\;\!\lambda_k}(r)\;=\;r$\ .\ \ If \ $r\in\stackrel{\ \circ}{J}$ ,\ \ then \ $p\;\!
\big(\,f^{\,k^2-1}(r,s)\,\big)\,\in\,\stackrel{\ \circ}{J}$ ,\ \ which with  \vspace{1.5mm}
\ $\xi_{\,n\;\!\lambda_k}(r)\,\in\,J$\ \ implies\ \ $\mbox{\bf\large$|$}\,\xi_{\,n\;\!\lambda
_k}(r)-p\;\!\big(\,f^{\,k^2-1}(r,s)\,\big)\mbox{\bf\large$|$} < 2$\ . \,Thus\,,\ \,no
matter whether \;$r\in\;\!\partial J$ \,or \;$r\in\stackrel{\ \circ}{J}$\,, \ we have
$$\mbox{\bf\large $|$}\,\xi_{\,n\;\!\lambda_k}(r)\,-\,p\;\!\big(\,f^{\,k
^2-1}(r,s)\,\big)\,\mbox{\bf\large$|$\,$/$}\,\big(2k+7\:\!\big)\,<\;2\,\mbox{\bf\large
$/$}(2k+7\:\!)\;=\ 2\,\mbox{\bf\large $/$}\big(\,2(\:\!k+1)+5\,\big),
$$
 which with \,(\ref{eq:2.7}) \,and \,(\ref{eq:2.8}) \,implies
\begin{equation}\label{eq:2.11}
 |\,pf(x)-p(x)\,|\ <\ 2\,\mbox{\bf\large$/$}\big(\,2(\:\!k+
1)+5\,\big) , \hspace{7mm}\mbox {for \ any} \ \ x\,\in\,X_k\ .
\end{equation}
Clearly\,, \ after the level linear extension\,,\ \ for \,any\ \ $x\,
\in\,J\times S_k$ , \ ({\ref{eq:2.11}}) \ still \,holds . \;Specially\,, \ noting \,$t_{\:\!(k+1)^2-
1}\in\,S_k$\,, \ we have \,$J_{\:\!t_{\:\!(\:\!k+1)^2-1}}\,\subset\;J\times S_k$\,. \,Thus\,,\
\,for the homeomorphism \,$f_{\ref{eq:2.5}}$ , \,the condition \,(C\,.\,3\,.\,$k+1$) \,holds .
\medskip

In addition , \ after the vertical linear extension ,\ \ for \,any \ $x\,
\in\,D_{k^2}^{\,(k+1)^2-1}$\,,\ \ by \,(C.3.$k$) \,and \,(\ref{eq:2.11}) \,( for all \ $x\,\in
\,J\times S_k$ ) \,we obtain
$$
|\,pf(x)-p(x)\,|\ <\ \,\max\,\big\{\;\!2/(2k+5)\ ,\;2\,\mbox{\bf\large$/$}\big(\,2(\:
\!k+1)+5\,\big)\;\!\big\}\ =\ \,2\,\mbox{\bf\large$/$}\big(\;\!2(\:\!k+1)+3\,\big)\ .
$$
This with\ \,(C.2.$k$)\ \,implies that\,,\ \,for the homeomorphism\ \,(\ref{eq:2.5})\,,\ \,the
condition \,(C\,.\,2\,.\,$k+1$) \,also holds .

\medskip
Therefore\,, \,by induction\,, \,we obtain a homeomorphism
\begin{equation}\label{eq:2.12}
f_{\ref{eq:2.12}}\;=\ f\;|\;D_0^{\,\infty}\ :\ \ D_0^{\,\infty}\ \rightarrow\ \,D_1^{\,\infty},
\end{equation}
 which satisfies the conditions\ \,(C.1.$k$)\;,\ \,(C.2.$k$)\ \ and\,\
(C.3.$k$)\ \ for all \ $k\,\in\,\mathbb N$ , \ and \,from these conditions \,we can directly
extend \,the \,homeomorphism \,$f_{\ref{eq:2.12}}$ \,to \,a \,homeomorphism
\begin{equation}\label{eq:2.13}
f_{\ref{eq:2.13}}\;=\ f\;|\ \overline{\;\!\!D_0^{\,\infty}\!}\ \;:\ \;\overline{\;
\!\!D_0^{\,\infty}\!}\ \ \to\ \ \overline{\;\!\!D_1^{\,\infty}\!}
\end{equation}
 by putting \ $f(x)\,=\,x$ \ for any \ $x\,\in\,J_1$ .
\medskip

As  mentioned above\,, \,in the domain \,$\overline{\;\!\!D_0^{\,\infty}\!}$ , \ we
will replace \,$f_{\ref{eq:2.13}}$\ \,by \,$f$\,, \,even if the definition of the entire\ \,$f
:J^{\,2}\rightarrow\,J^{\,2}$\ \ has  not yet  been given\,. \,Specially\,,\ \,for any \,$x\:
\!\in\,\overline{\;\!\!D_0^{\,\infty}\!}$\,, \ we can write \ $\omega(x,f)$ \ for \ $\omega(x\,,f_{\ref{eq:2.13}
})$\ , \ since \ $\omega(x,f)\,=\;\omega(x,f_{\ref{eq:2.13}})$ ,\ \ no matter how \ $f\;\mbox{\large$|
$}\ D_{-\,\infty}^{\,-1}\ :\;\,D_{-\,\infty}^{\,-1}\ \rightarrow\;\,D_{-\,\infty}^{\,0}$ \ is defined\,.

\medskip
{\bf Claim 2.5.1.}\ \ \ for \,any \ $n\,\in\;\mathbb N$ \ and \,any \,given \ $s
\in V_n$ ,\ \ it holds that \ $\mbox{\large$\omega$}_{sf}\;=\ \mbox{\large$\omega$}_n$ .

\vspace{3mm}{\bf Proof of Claim 2.5.1.} \ \ Consider any given \ $r\in R$ . \ Take an
integer \ $j\,\geq\,n$ \ such that \ $r\,\in\,R_j$ \ and \ $s\,\in\,V_{nj}$ . \ Then\,
$(r,s)\,\in\,R_k\times V_{nk}$ \,\ for any integer \ $k\,\geq\,j$\ .\ \ By \,(\ref{eq:2.6}) \,and\,
(\ref{eq:2.7}) \,we can easily verify that \ $\omega_f(r,s)\,=\;\big[\,\xi_{\,n1}(r)\;,\,\xi_{\,n2
}(r)\,\big]\,\times\,\{1\}$ . \,This \,with \,(\ref{eq:2.1}) \,and \,(\ref{eq:2.2}) \,implies \,$\mbox{
\large$\omega$}_{sf}(r)\;=\;\,\mbox{\large$\omega$}_n(r)$ . \,Thus we have \ $\mbox{\large$\omega
$}_{sf}\,|\,R\;=\ \mbox{\large$\omega$}_n\,|\,R$ .
\medskip

Further\,, \,consider any given \ $t\in J-R$\ .\ \,Since \,$R$ \,contains
all discontinuous points of \,$\xi_{\,n1}$ \ and \ $\xi_{\,n2}$ , \ it follows that\,
$t$ \,is a continuous point both of \ $\xi_{\,n1}$ \ and of \ $\xi_{\,n2}$ .\ \ Since
\,$R$ \,is a dense subset of \,$J$\,, \,\ $\partial J\subset R$\ , \ and since \ $\xi_{\,n1}$ \
and \ $\xi_{\,n2}$ \ are increasing , \ there exist \ \ $t_i\,\in\,(\,t\;,\;t+1/\,i\;
]\,\cap\,R$ \ \ and \ \ $\tau_i\,\in\,[\;t-1/\,i\;,\;t\,)\,\cap\,R$ \ \ for each \ $i
\,\in\,\mathbb N$\ \ such that
\begin{equation}\label{eq:2.14}
\xi_{\,n1}(t)-1/\,i\ < \ \xi_{\,n1}(\tau_i)\ \le\ \xi_{\,
n1}(t)\ \le\ \xi_{\,n1}(t_i)\ <\ \xi_{\,n1}(t)+1/\,i
\end{equation}
and
\begin{equation}\label{eq:2.15}
\xi_{\,n2}(t)-1/\,i\ < \ \xi_{\,n2}(\tau_i)\ \le\ \xi_{\,
n2}(t)\ \le\ \xi_{\,n2}(t_i)\ <\ \xi_{\,n2}(t)+1/\,i\ .
\end{equation}
On the other hand\ ,\ \,from \,Lemma \ref{rising-prop2}\ \,we know that there exist
increasing functions \ $\psi_{\,n1}\,:\,J\,\rightarrow\;J$\ \ and \ $\psi_{\,n2}\,:\,J\,\rightarrow\;J$ \
such that\,,\ \ \,for \ any \ $r\,\in\,J$\,, \ \,it holds that
\begin{equation}\label{eq:2.16}
\psi_{\,n1}(r)\,\le\,\psi_{\,n2}(r)   \hspace{6mm}   \mbox{and}    \hspace{6mm}
\mbox{\large$\omega$}_{sf}(r)\, =\;\big[\,\psi_{\,n1}(r)\;,\,\psi_{\,n2}(r)\,\big]\,\times\,
\{1\}\ .
\end{equation}
Noting\ \,\ $\mbox{\large$\omega$}_{sf}\,|\,R\; = \ \mbox{\large$\omega$}_n\,
|\,R$\ ,\ \ from \ (\ref{eq:2.16}) \ and \ (\ref{eq:2.2}) \ we get\,,\ \ \,for \ any \ $i\in\mathbb N$\,,
$$
\psi_{\,n1}(\tau_i)\ =\ \xi_{\,n1}(\tau_i)\ ,\ \,\ \psi_{\,n1}(t_i)\ =\ \xi_{\,n1}(t_i) \
,\ \,\ \psi_{\,n2}(\tau_i)\ =\ \xi_{\,n2}(\tau_i)\ ,\ \,\ \psi_{\,n2}(t_i)\ =\ \xi_{\,n2}
(t_i)\ ,
$$
which with \ \,$\psi_{\,n1}(\tau_i)\,\leq\,\psi_{\,n1}(t)\,\leq\,\psi_{\,n1}(t_i)$ , \ \,$\psi_
{\,n2}(\tau_i)\,\leq\,\psi_{\,n2}(t)\,\leq\,\psi_{\,n2}(t_i)$ \ \,and \ \,(\ref{eq:2.14}) , \ (\ref{eq:2.15})
\ \,imply \ \,$\psi_{\,n1}(t)\ =\ \xi_{\,n1}(t)$\ \ \,and \ \,$\psi_{\,n2}(t)\ =\ \xi_{\,
n2}(t)$ . \ Thus we have \ $\mbox{\large$\omega$}_{sf}(t)\;=\;\,\mbox{\large$\omega$}_n(t)$ \
and hence \ $\mbox{\large$\omega$}_{sf}\,=\ \mbox{\large$\omega$}_n$\ . \ Claim 2.5.1 \,is\,\
proved\,.
\medskip

\setlength{\unitlength}{1mm}
\begin{picture}(100,100)
\put(0,50){\vector(1,0){100}}
\put(50,5){\vector(0,1){90}}
\put(10,10){\line(1,0){80}}
\put(10,10){\line(0,1){80}}
\put(10,90){\line(1,0){80}}
\put(90,10){\line(0,1){80}}
\put(10,30){\line(1,0){80}}
\put(10,70){\line(1,0){80}}
\put(10,20){\line(1,0){80}}
\put(10,80){\line(1,0){80}}
\put(10,15){\line(1,0){80}}
\put(10,85){\line(1,0){80}}
\put(8,92){$v_1$}
\put(91,91){$v_2$}
\put(8,7){$v_3$}
\put(91,7){$v_4$}
\put(10,10){\circle*1}
\put(10,90){\circle*1}
\put(90,10){\circle*1}
\put(90,90){\circle*1}
\put(91,46){\scriptsize$1$}
\put(5,46){\scriptsize$-1$}
\put(47,91){\scriptsize$1$}
\put(45,6){\scriptsize$-1$}
\put(47,46){\scriptsize$0$}
\put(45,67){\scriptsize{$1/2$}}
\put(45,77){\scriptsize{$3/4$}}
\put(42,27){\scriptsize{$-1/2$}}
\put(42,17){\scriptsize{$-3/4$}}
\put(26,60){\small$D_1$}
\put(26,40){\small$D_0$}
\put(26,24){\small$D_{-1}$}
\put(26,16){\small$D_{-2}$}
\put(26,75){\small$D_2$}
\put(26,81){\small$D_3$}
\put(28,86){\circle*0.9}
\put(28,87){\circle*0.9}
\put(28,88){\circle*0.9}
\put(28,14){\circle*0.9}
\put(28,13){\circle*0.9}
\put(28,12){\circle*0.9}
\put(50,60){\circle*1}
\put(47,59){s}
\put(50.5,59){\Large$\}$}
\put(54,59){$V_n$}
\put(80,59){\circle*1}
\put(75,55){$(r_j,s)$}
\put(70,75){\circle*1}
\put(72,73){$f(r_j,s)$}
\put(60,82){\circle*1}
\put(62,81){\small$f^2(r_j,s)$}
\put(80,50){\circle*1}
\put(79,46){$r_j$}
\put(30,90){\circle*1}
\put(60,90){\circle*1}
\put(22,93){\tiny{$(\xi_{n1}(r_j),1)$}}
\put(55,93){\tiny{$(\xi_{n2}(r_j),1)$}}
\put(40,87){.....}
\put(42,-2){Figure 2.1}
\end{picture}

\medskip
Similarly\;, \ we can construct a homeomorphism
\begin{equation}\varphi\ :\ \;\overline{\;\!\!D_{-\,\infty}^{\,1}\!\!\!\!}\ \ \ \,\rightarrow\ \
\overline{\;\!\!D_{-\,\infty}^{\,0}\!\!\!\!}\
\end{equation}
such that \

\vspace{2mm}(1) \ \ $\varphi\;|\ D_1\,\cup\,J_{-\,1}\;=\;f_{02}^{-1}\;|\ D_1\,\cup\,J_{-\,
1}\;:\ D_1\,\cup\,J_{-\,1}\;\,\rightarrow\ D_0\,\cup\,J_{-\,1}$\, ;

\vspace{2mm}(2)\ \ \ $\varphi\;\!(J_s)\;=\;f_{02}^{-1}(J_s)$ ,\ \ \ for \,any \ $s\;\in\;[
\;-1\;,\,1/\;\!2\;]$\, ;

\vspace{2mm}(3)\ \ \ $\omega_\varphi(r,s)\;=\;\mbox{\large$\alpha$}_{j\,}(r)$ ,\ \ \ for \,any
\ $r\,\in\,J$ , \ any $j\,\in\,\mathbb N\,''$ , \ and \,any \ $s\,\in\,W_j$\, .

\vspace{2mm}\noindent Define \ $f:J^{\,2}\rightarrow\,J^{\,2}$ \ \ by \ \
$$
f\;|\ \,\overline{\;\!\!D_0^{\,\infty}\!}\;\ =\;\,f_{\ref{eq:2.13}}\ \,,     \hspace{9mm}     \mbox{and}
\hspace{9mm}    f\;|\ \,\overline{\;\!\!D_{-\,\infty}^{\,-1}\!\!}\ \ =\;\,\varphi^{\,-\,1}\;|\ \,\overline
{\;\!\!D_{-\,\infty}^{\,-1}\!\!}\ \ \,.
$$
Then\ \,$f$\ \,is a homeomorphism which satisfies the conditions mentioned in Theorem
\ref{main1'}\;, \;and the proof is complete. \end{proof}

\vspace{5mm}

\section{Bounded and unbounded orbits of homeomorphisms on $\mathbb R^2$}

In this section, we will use  Theorem \ref{main1'} to construct a homeomorphism on the plane which
illustrates an interesting phenomenon: points of bounded orbits can surround points of divergent orbits.

\medskip
Let\ \ $\Psi:J^{\,2}\,\rightarrow\,J^{\,2}$\ \ be \,the\,\ {\it level reflect}\ \ and\ \ let\ \
$\Psi_v:J^{\,2}\,\rightarrow\,J^{\,2}$\ \ be \,the \,{\it vertical reflect} \ defined \,by
\begin{equation}\label{eq:3.1}
\Psi(r,s)\,=\,(\:\!-\,r,s) \hspace{5mm} \mbox{and} \hspace{5mm} \Psi
_v(\vspace{1mm}r,s)\,=\,(r,\,-s) \hspace{5mm} \mbox{for \ any} \ (r,s)\in\,J^{\,2}\,.
\end{equation}

\begin{lem}\label{6-points} Let \ $K\,=\,[\,1/3\,,\,1/2\,]$ , \ and let the rectangle\ \
$F\,=\;[\,-1/2\,,\,1/2\,]\times K$\;. \ Write \ $L_1\,=\,\{-1/2\}\times\stackrel{\ \circ}{K}$ , \ and
\ $L_2\,=\,\partial F-\,L_1$\ . \ Let \ \,$u_1\,,\ \cdots\,,\ u_6$
\ \ be \ six \ points \ in \ \,$J\,\times\;\partial J$ \ \ with
$$
u_1=(\,-1/2\,,\,1)\;,\ \ u_2=(\,0\,,\,1)\;,\ \ u_3=(\,1/2\,,\,1)\;,\ \ \,{\rm and}\ \
\ u_{\,i\,+\,3}=\;\!\Psi_v(u_i)\ \ \ {\rm for}\ \ \,i\,\in\,\mathbb N_{\:\!3}
$$
( see \,Fig.\;3.1 \,below)\;. \ Then there exists a \ normally rising homeomorphism \ $f:J^{\,2}\to J^{\,2}$ \ such that

(1) \ \ \ $\omega(x,f)\,=\,\{u_1\}$\ \ \ and\ \ \ $\alpha(x,f)\,=\,\{u_4\}$\ \ \ for any\ \ \
$x\in L_1$\ ,

(2) \ \ \ $\omega(y,f)\,=\,\{u_2\}$\ \ \ and\ \ \ $\alpha(y,f)\,=\,\{u_5\}$\ \ \ for any\ \ \
$y\in \stackrel{\ \circ}{F}$\ ,\ \ \ \ and

(3) \ \ \ $\omega(z,f)\,=\,\{u_3\}$\ \ \ and\ \ \ $\alpha(z,f)\,=\,\{u_6\}$\ \ \ for any\ \ \
$z\in L_2$\;.
\end{lem}

\begin{proof}
Let \ ${\mathcal{V}}=\{V_1\,, V_2\,, V_3\,\}$ \ with \ $V_1
=\{1/3\}\,, \ \ V_{2}=\stackrel{\circ}{K}\,, \ \ V_3=\{1/2\}$\,. \ Let \ $\{\,\mbox{\large$\omega$}_i: J\rightarrow
{\mathcal{A}}\,\}_{i\,=1}^{\,3}$ \ and \ $\{\mbox{\large$\alpha$}_i: J \rightarrow{\mathcal{A}}'\,\}_{i\,
=1}^{\,3}$ \ be endpoint preserving and increasing maps\,,\ \,which satisfy

\vspace{2mm}(a)\ \ \ $\mbox{\large$\omega$}_1(r)\;=\;\mbox{\large$\omega$}_3(r)\;=\;\{u_3\}$\
,\ \ for any \ $r\,\in\,[\,-1/2\,,\,1/2\,]$\ ;

\vspace{1mm}(b)\ \ \ $\mbox{\large$\omega$}_2(-1/2)\;=\;\{u_1\}\ ,\ \ \mbox{\large$\omega$}_2
(1/2)\;=\;\{u_3\}$\ ,\ \ and\ \ $\mbox{\large$\omega$}_2(r)\;=\;\{u_2\}$ \ for any \ $r\,
\in\,(\,-1/2\,,\,1/2\,)$\ ;

\vspace{2mm}(c)\ \ \ $\mbox{\large$\alpha$}_1(r)\;=\;\mbox{\large$\alpha$}_3(r)\;=\;\{u_6\}$\
,\ \ for any \ $r\,\in\,[\,-1/2\,,\,1/2\,]$\ ;

\vspace{1mm}(d)\ \ \ $\mbox{\large$\alpha$}_2(-1/2)\;=\;\{u_4\}\ ,\ \ \mbox{\large$\alpha$}_2
(1/2)\;=\;\{u_6\}$\ ,\ \ and\ \ $\mbox{\large$\alpha$}_2(r)\;=\;\{u_5\}$ \ for any \ $r\,
\in\,(\,-1/2\,,\,1/2\,)$\ .

\noindent Then by Theorem \ref{main1'}, there exists a normally rising homeomorphism\ \ $f:J^{
\,2}\rightarrow J^{\,2}$ \ such that \ $\mbox{\large$\omega$}_{sf}\,=\,\mbox{\large$\omega$}_i$ \ and
\ $\mbox{\large$\alpha$}_{sf}\,=\,\mbox{\large$\alpha$}_i$ \ for any \ $i \in\mathbb N_3$\ \ and
any \ $s\in V_i$\;. \ Such an \,$f$\ \ will satisfies the requirements.
\,The proof is complete\,.
\end{proof}

Let \ $X$ \ and \ $Y$ \ be topological spaces\,, \,and \ $\beta:X\,\rightarrow\,X$ \
and \ $\gamma:Y\,\rightarrow\,Y$ \ be continuos maps\,.\ \,If there exists a continuos surjection
(\,resp. a homeomorphism) \ $\eta:X\,\rightarrow\,Y$\ \ such that \ $\eta\beta\;=\;\gamma\eta$\ , \
then\ \ $\beta$\,\,\ and\ \ $\gamma$\ \ are said to be\ \,{\it topologically semi-conjugate}
(\,resp.\ \,{\it topologically conjugate})\ ,\ \,and\ \ $\eta$\ \ is called a\ \,{\it
topological semi-conjugacy} (\;resp. a {\it topological conjugacy})\,\ from \ $\beta$\ \
to \ $\gamma$ . \ The following lemma is well known\,, \,however\,, \,for convenience\,,\
\,we still give a short proof\,.

\begin{lem} \label{semi-conj} Let\ \ $\beta:X\,\rightarrow\,X$\ \ and \ $\gamma:Y\,\rightarrow\,Y$ \
with a topological semi-conjugacy \ $\eta:X\,\rightarrow\,Y$\ \ be as above\;. \ If \;both\ \
\,$X$\ \,and\ \ \,$Y$\ \,are compact metric spaces\;, \,then \

\vspace{1mm}$(1)$\ \ \ For \,any \ $x\;\in\;X$\,, \ the \ $\omega$-limit set \,\ $\omega\big(
\eta(x)\,,\,\gamma\big)\,=\,\;\eta\big(\omega(x\,,\,\beta)\big)\ ;$

\vspace{1mm}$(2)$ \ \ If \;both \ $\beta$ \ and \ $\gamma$ \,are \,homeomorphisms\,, \ then
, \ for \,any \ $x\;\in\;X$\,, \ the \,$\alpha$-limit set\ \ \,$\alpha\big(\eta(x)\,,\,\gamma\big)
\,=\,\;\eta\big(\alpha(x\,,\,\beta)\big)$\ .
\end{lem}

\begin{proof}\ \ $(1)$\ \,\ For any given \ $x\,\in\,X$\ , \ let \ $y\;=\;
\eta(x)$ . \ For any \ $n\,\in\,\mathbb N$ , \ write \ $x_n\;=\;\beta^{\;n}(x)$\;, \ and\ \
$y_n\;=\;\gamma^{\;n}(y)$ . \ If \,some point \ $w\,\in\,\omega(x\,,\,\beta)$\ , \ then there is
a sequence \ $n_1<n_2<\cdots$ \ in \ $\mathbb N$ \ such that \ \,$\lim_{\,i\,\rightarrow\,\infty
}\,x_{n_i}\;\!=\;w$ . \ By the continuity of \ $\eta$ , \ we have \ \,$\lim_{\,i\,\rightarrow
\,\infty}\,y_{n_i}=\;\eta(w)$\;, \ which means \ $\eta(w)\,\in\,\omega(y\,,\,\gamma)$ , \ and
hence \ $\eta\big(\omega(x\,,\,\beta)\big)\;\subset\;\omega\big(\eta(x)\,,\,\gamma\big)$ .
\vspace{2mm}Conversely\,, \,if \,some point \ $u\,\in\,\omega(y\,,\,\gamma)$\ ,\ \ then there
is a sequence \ $n_1<n_2<\cdots$\ \ in\ \ $\mathbb N$\ \ such that \ \,$\lim_{\,i\,\rightarrow\,
\infty}\,y_{n_i}\;\!=\;u$\ .\ \ Since \ $X$\ \ is compact\,,\ \ there is a point \ $w
\,\in\,X$\ \ and a subsequence\ \ $m_1<m_2<\cdots$\ \ of\ \,the sequence\ \ $n_1<n_2<
\cdots$\ \ such that \ \,$\lim_{\,i\,\rightarrow\,\infty}\,x_{m_i}\;\!=\;w$\;,\ \ which means
\ $w\,\in\,\omega(x\,,\,\beta)$\ \ and leads to \ \,$\lim_{\,i\,\rightarrow\,\infty}\,y_{m_i}\;\!=\;
\eta(w)\;=\;u$\;.\ \ Thus we have \ $\omega\big(\eta(x)\,,\,\gamma\big)\;\subset\;\eta\big(\omega(x\,,
\,\beta)\big)$ .

\vspace{2mm}$(2)$ \ \ If \;both \ $\beta$\ \ and \ $\gamma$\ \,are homeomorphisms\,, \ then
\ $\eta$\ \ is also a topological semi-conjugacy from \ $\beta^{\,-1}$ \ to \ $\gamma^{\,-1}
$\,,\ \ and from the conclusion (1) we get

\vspace{1mm}\hspace{16mm} $\alpha\big(\eta(x)\,,\,\gamma\big)\,=\ \omega\big(\eta(x)\,,\,\gamma^{\,-1
}\big)\,=\ \eta\big(\omega(x\,,\,\beta^{\,-1})\big)\,=\ \eta\big(\alpha(x\,,\,\beta)\big)$\ .
\end{proof}

 The following theorem is  well known\,, \,which is an  equivalent form of
the Sch\"{o}nflies theorem ( see e.g. \,\cite[p.72]{Mo77}\,)\,.

\begin{thm}\label{sch}
For any disks \ $E$ \ and \ $G$ \ in \ $\mathbb R^{\,
2}$\,,\ \ there exists a homeomorphism \ $\zeta:\mathbb R^{\,2}\rightarrow\mathbb R^{\,2}$\ \ such that\ \ $\zeta
(G)\,=\,E$\,.
\end{thm}

 Let \ $d$ \ be the Euclidean metric on \ $\mathbb R^{\,2}$.\ \ For any homeomorphism\ \
$h:\mathbb R^{\,2}\to\mathbb R^{\,2}$\ \ and any \ $x\,\in\,\mathbb R^{\,2}$,\ \ the orbit \ $O(x,h)$ \ is
said to be \,{\it positively bounded} (\,resp.\ \,{\it
negatively bounded}\,)\ \,if \ $O_+(x,h)$ (\,resp.\ \ $O_-(x,h)$ )\ \ is bounded\,. \
If \ $O(x,h)$ \ is not positively bounded (\,resp.\,\ not negatively bounded\,)\,, \,
then\,\,\ $O(x,h)$\,\,\ is said to be\ \,{\it positively unbounded}\ (\,resp.\ \,{\it
negatively unbounded}\,)\,.\ \
\medskip

The following lemma is clear\,.

\begin{lem}\label{bound} Let\ \ $h:\mathbb R^{\,2}\rightarrow \mathbb R^{\,2}$\ \ be a  homeomorphism\,,\,\
and \ $x\,\in\,\mathbb R^{\,2}$.\ \ Then

$(1)$\ \ \ $O(x,h)$ \ is \,positively \,divergent \,if \,and \,only \,if\ \ \,$\omega(x,h
)\,=\ \emptyset\ ;$

$(2)$\ \ \ If\ \ \,$O(x,h)$ \ is \,positively \,unbounded\;,\ \; then\ \ \,$\omega(x,h)\,
\neq\;\emptyset$\ \ if \,and \,only \,if\ \ \,$\omega(x,h)$\ \ is \,an \,unbounded \,set $;$

$(3)$\ \ \ If\ \ \,$\omega(x,h)$ \ is \,a \,nonempty \,bounded \,set\,, \,then \ $O(x,h)$
\ is \,positively \,bounded.
\medskip

In the negative direction of the orbit\,\,\,\,$O(x,h)$\ ,\ \,we also have
similar conclusions\,.
\end{lem}

\begin{thm}\label{main2'}
Let \ $E$\ \ be a disk in \ $\mathbb R^{\,2}$\,.\,\
Then there exists a homeomorphism \ $h:\mathbb R^{\,2}\rightarrow\mathbb R^{\,2}$\ \;such that\,, \ for any
\ $x\in\,\partial E$\,,\ \ the \,orbit \ $O(x,h)$\ \;is \,bounded\;, \,but \,for any\ \ $y
\in\,\stackrel{\ \circ}{E}$\,,\ \ the \,orbit \ $O(y,h)$ \;is \,doubly\,-\;\!divergent\,.
\end{thm}

\begin{proof}
Continue to use the all notations in Lemma \ref{6-points}\,. \,Let
\ $v_1\,,\,\cdots\,,\,v_4$ \ and \ $w_1\,,\,\cdots\,,\,w_8$\ \ be points\ \,in \ \,$J
\times\,\partial J$ \ \,with \ \ ( see Fig. 3.1)
$$
v_1\,=\,(\,-1\,,\,1)\;,\ \ \ \ \ \ \,v_2\,=\,(\,1\,,\,1)\;,\ \ \ \ \ \ \,v_3\,=\,(\,-
1\,,\,-1)\;,\ \ \ \ \ \ \,v_4\,=\,(\,1\,,\,-1)\;,
$$
$$
w_1\,=\,(\,-3/4\,,\,1)\;,\ \ \ \ \,w_2\,= \,(\,-1/4\,,\,1)\;,\ \ \ \ \,w_3\,=\,(\,1/4
\,,\,1)\;,\ \ \ \ \,w_4\,=\,(\,3/4\,,\,1)\;,                            \vspace{-0mm}
$$
and \ $w_{\,i\,+\,4}\,=\Psi_v(w_i)$\ \ for \ $i\,\in\,\mathbb N_{\:\!4}$\,. \ \,Let \ \,$x
_1\,,\ \cdots\,,\ x_6$ \ \,be points \,in \ $\stackrel{\ \circ}{J}$$^{\,2}$\ \ with
\vspace{-1mm}$$
x_1\,=\,(\,-1/2\,,\,3/4)\;,\ \ \ \ \ \ \ \ \ x_2\,= \,(\,0\,,\,3/4)\;,\ \ \ \ \ \ \ \
\ x_3\,=\,(\,1/2\,,\,3/4)                                               \vspace{-1mm}
$$
and\ \ \,$x_{\,i\,+\,3}\,=\Psi_v(x_i)$ \ \,for\ \;\,$i\,\in\,\mathbb N_{\:\!3}$\,.
For any points \ $y_1\;,\ y_2\;,\ \cdots\;,\ y_n$ \ in \ $\mathbb R^{\,2}$\ \ with \ $n\geq 2$\,,\
\ denote by \ $[\,y_1\,,\,y_2\,,\,\cdots\,,\,y_n\,]$ \ or by \ $[\,y_1\,y_2\,\cdots\,
y_n\,]$ \ the smallest convex set containing \ $y_1\,,\,y_2\,,\,\cdots\,,\,y_n$\;.\ \
\,Clearly\,, \,there is a continuous map \ \ $\xi:J^{\,2}\,\rightarrow\,J^{\,2}$ \ satisfying
the following conditions\ :

\vspace{3mm}(a) \ \ $\xi\,|\,\big(\,J\times[\,-1/2\,,\,1/2\,]\,\big)\,\cup\,\big(\,\{
\,-1\,,\,0\,,\,1\}\times\,J\,\big)$\ \ is the identity map\ ;

\vspace{2mm}(b)\ \ \ \ $\xi(u_3)\;=\;x_3\;,\ \ \ \xi(w_3)\;=\;\xi(w_4)\;=\;u_3$\;,\ \
\ \ \ and\ \ \ \ \ $\xi\,|\,[\,u_2\,w_3\,]\;,\ \ \ \xi\,|\,[\,w_3\,u_3\,]\;,\ \ \ \xi
\,|\,[\,u_3\,w_4\,]$\ \ \ and\ \ \ $\xi\,|\,[\,w_4\,v_2\,]$\ \ \ are linear\ ;

\vspace{2mm}(c)\ \ \ $\xi\;\!\Psi\,=\;\Psi\,\xi$\ ,\ \ \ and\ \ \ $\xi\;\!\Psi_v\,=\;\Psi_v\,
\xi$ , \ \,where \ $\Psi:J^{\,2}\,\rightarrow\,J^{\,2}$ \ is \,the \;level \,reflect\,, \ and \
$\Psi_v:J^{\,2}\,\rightarrow\,J^{\,2}$\ \ is \,the \,vertical \,reflect\,, \ defined\ \,as \,in
\,(\ref{eq:3.1})\;.

\vspace{2mm}(d) \ \ $\xi\,|\,\stackrel{\ \circ}{J}$$^{\,2}$\ \,is an injection\,,\ \ and\ \
$\xi(\,\stackrel{\ \circ}{J}$$^{\,2}\,)\,=\,\stackrel{\ \circ}{J}$$^{\,2}\,-\,[\,u_1 x_1\,]\,-\,[
\,u_3x_3\,]\,-\,[\,u_4x_4\,]\,-\,[\,u_6x_6\,]$\ .

\vspace{4mm}Let\ \ $f:J^{\,2}\rightarrow J^{\,2}$\ \ be the homeomorphism given in Lemma \ref{6-points}\
.\ \ Define a map\ \ $g:J^{\,2}\to J^{\,2}$\ \ by\ \ $g\;=\;\xi\,f\,\xi^{\;-1}$\,.\ \
Note that\,,\ \,if\ \ \,$x\,\in\;[\,u_1x_1\,]\,\cup\;[\,u_3x_3\,]\,\cup\;[\,u_4x_4\,]
\,\cup\;[\,u_6x_6\,]\,-\,\{\,x_1\,,\,x_3\,,\,x_4\,,\,x_6\,\}$\;,\ \ then \ $\xi^{\;-1
}(x)$ \ contains two points\,,\ \,but \ $\xi\,f\,\xi^{\;-1}(x)$ \ still contains only
one point\,.\ \,Thus \ $g$ \ is well defined\,. \ It is easy to see that \ $g$ \ is a
bijection\,,\ \,and \ $g$\ \ is continuous\,. \,Thus \ $g:J^{\,2}\to J^{\,2}$ \ is an
orientation preserving homeomorphism\,.\ \ Moreover\,, \ from \ $g\;=\;\xi\,f\,\xi^{\;
-1}$ \ we obtain\ \ $g\xi\;=\;\xi\,f$\,, \,this means that \ $f$\ \ and\ \ $g$\ \ are
topologically semi-conjugate\,, \,and \ $\xi$ \ is a topological semi-conjugacy from\
\ $f$\ \ to \ $g$\ .\ \ By \,Lemmas \ref{6-points} \,and \, \ref{semi-conj} \,we get

\vspace{2mm}{\bf Claim 3.4.1.}\ \ \ (1)\ \ \ \ $\omega(x,g)\,=\,\{x_1\}$\ \ \ and \ \ $\alpha
(x,g)\,=\,\{x_4\}$\ \ \ \ for any\ \ \ $x\in L_1$\ ,\ \ \

(2)\ \ \ \ $\omega(y,g)\,=\,\{u_2\}$\ \ \ and \ \ $\alpha(y,g)\,=\,\{u_5\}$\ \ \ \ for any\ \
\ $y\in \stackrel{\ \circ}{F}$\ ,\ \ \ \ and

(3)\ \ \ \ $\omega(z,g)\,=\,\{x_3\}$\ \ \ and \ \ $\alpha(z,g)\,=\,\{x_6\}$\ \ \ \ for any\ \
\ $z\in L_2$\;.\ \ \

\vspace{4mm}Define a homeomorphism \ $\psi\,:\,\stackrel{\ \circ}{J}$$^{\,2}\,\rightarrow\,\mathbb R^{\,
2}$ \ by \ $\psi(r,s)\ =\ \big(\,{\rm tg}(\;\!\pi r/2)\ ,\;{\rm tg}(\;\!\pi s/2)\,\big)
$\,,\ \ for any \ $(r,s)\,\in\,\stackrel{\ \circ}{J}$$^{\,2}$\,.\ \ Write\ \ $G\,=\;\psi(F)$\
.\ \,Then \ $G\,=\,[\,-1,1\,]\times[\,\sqrt{3}\,/\,3\;,1\,]$ \ is also a rectangle\,,
\,and \ $\psi(L_1)\,=\,\{-1\}\times(\,\sqrt{3}\,/\,3\;,1\,)$\ , \ $\psi(L_2)\,=\;\partial G\,-
\,\psi(L_1)$ . \ By Theorem \ref{sch}\,, \ there exists a homeomorphism \ $\zeta:\mathbb R^{\,2}\rightarrow\mathbb R^{\,
2}$\ \ such that \ $\zeta(G)\,=\,E$ .\ \ Let \ \,$h\;=\ \zeta\psi\,g\,\psi^{\,-1}\zeta^{\,-1}\,:\,
\mathbb R^{\,2}\,\rightarrow\,\mathbb R^{\,2}$\,.\ \ Then \ $h$\ \ is also an orientation preserving
homeomorphism\,, \ which is topologically conjugate to \ $g$ ,\ \ and \ $\zeta\psi$ \ is a
topological conjugacy from \ $g$ \ to \ $h$. \ By \ Claim 3.4.1 \,and \,Lemma \ref{semi-conj}\,,
\ \;we have

\vspace{2mm}{\bf Claim 3.4.2.}\ \ (1)\ \ $\omega(x,h)\,=\,\{\zeta\psi(x_1)\}$, \ and \ $\alpha(x,h
)\,=\,\{\zeta\psi(x_4)\}$\,,\ \ for  \vspace{1mm}  any\ \ $x\in\zeta\psi(L_1)$\;;

(2) \ \ $\omega(y,h)\,=\ \alpha(y,h)\,=\ \emptyset$\,, \ \ \ for any \ \ $y\;\in\ \zeta\psi(\stackrel{\ \circ}
{F})$\ ;                                                                 \vspace{1mm}

(3)\ \ \ $\omega(z,h)\,=\,\{\zeta\psi(x_3)\}$\,, \ \ and \ \ $\alpha(z,h)\,=\,\{\zeta\psi(x_6)\}$\,,\ \
\ \ for \,any\ \ $z\,\in\,\zeta\psi(L_2)$\;.

\vspace{4mm}Noting that\ \ $\stackrel{\ \circ}{E}\;=\ \zeta\psi(\stackrel{\ \circ}{F})$\; \ and \ $\partial
E\,=\,\zeta\psi(\partial F)\,=\,\zeta\psi(L_1)\,\cup\,\zeta\psi(L_2)$\ ,\ \ from\ \,Claim 3.4.2\ \,and\,\
Lemma \ref{bound} \ we see that the homeomorphism \ $h$ \ satisfies the requirement. \ The \,proof \,is \,complete\,.
\end{proof}

\setlength{\unitlength}{1mm}
\begin{picture}(100,100)
\put(0,50){\vector(1,0){100}}
\put(50,5){\vector(0,1){90}}
\put(10,10){\line(1,0){80}}
\put(10,10){\line(0,1){80}}
\put(10,90){\line(1,0){80}}
\put(90,10){\line(0,1){80}}
\put(10,10){\circle*1}
\put(20,10){\circle*1}
\put(30,10){\circle*1}
\put(40,10){\circle*1}
\put(50,10){\circle*1}
\put(60,10){\circle*1}
\put(70,10){\circle*1}
\put(80,10){\circle*1}
\put(90,10){\circle*1}
\put(10,90){\circle*1}
\put(20,90){\circle*1}
\put(30,90){\circle*1}
\put(40,90){\circle*1}
\put(50,90){\circle*1}
\put(60,90){\circle*1}
\put(70,90){\circle*1}
\put(80,90){\circle*1}
\put(90,90){\circle*1}
\put(30,20){\circle*1}
\put(50,20){\circle*1}
\put(70,20){\circle*1}
\put(30,80){\circle*1}
\put(50,80){\circle*1}
\put(70,80){\circle*1}
\put(20,64){$L_1$}
\put(25,64){$\rightarrow$}
\put(76,64){$L_2$}
\put(71,64){$\leftarrow$}
\put(55,64){$F$}
\put(8,92){\scriptsize$v_1$}
\put(18,92){\scriptsize$w_1$}
\put(28,92){\scriptsize$u_1$}
\put(38,92){\scriptsize$w_2$}
\put(46,92){\scriptsize$u_2$}
\put(58,92){\scriptsize$w_3$}
\put(68,92){\scriptsize$u_3$}
\put(78,92){\scriptsize$w_4$}
\put(88,92){\scriptsize$v_2$}
\put(8,7){\scriptsize$v_3$}
\put(18,7){\scriptsize$w_5$}
\put(28,7){\scriptsize$u_4$}
\put(38,7){\scriptsize$w_6$}
\put(46,7){\scriptsize$u_5$}
\put(58,7){\scriptsize$w_7$}
\put(68,7){\scriptsize$u_6$}
\put(78,7){\scriptsize$w_8$}
\put(88,7){\scriptsize$v_4$}
\put(27,77){\scriptsize$x_1$}
\put(47,77){\scriptsize$x_2$}
\put(67,77){\scriptsize$x_3$}
\put(27,22){\scriptsize$x_4$}
\put(46,22){\scriptsize$x_5$}

\thicklines
\put(30,80){\line(0,1){10}}
\put(50,80){\line(0,1){10}}
\put(70,80){\line(0,1){10}}
\put(30,10){\line(0,1){10}}
\put(30,10){\line(0,1){10}}
\put(50,10){\line(0,1){10}}
\put(70,10){\line(0,1){10}}
\put(30,62){\line(0,1){7}}

\linethickness{2pt}
\put(30,62){\line(1,0){40}}
\put(30,69){\line(1,0){40}}
\put(70,62){\line(0,1){7}}
\put(68,22){\scriptsize$x_6$}

\put(43,-2){Figure 3.1}
\end{picture}

\subsection*{Acknowledgements}
Jiehua Mai, Kesong Yan, and Fanping Zeng are supported by NNSF of China (Grant No. 12261006); Kesong Yan is also supported by NNSF of China (Grant No. 12171175); Enhui Shi is supported by NNSF of China (Grant No. 12271388).

\end{document}